\documentclass[preprint]{elsarticle}
\usepackage{color}
\usepackage{stix}
\usepackage[utf8]{inputenc}
\usepackage{amsmath}   
\usepackage{algorithm}
\usepackage{algorithmic}
\usepackage{pdflscape}

\usepackage[left=1in,
            right=1.2in,
            top=1in,
            bottom=1in]{geometry}
\usepackage{subcaption}
\usepackage{graphicx}  
\usepackage{listings}
\usepackage{enumerate}
\usepackage{array}     
\usepackage{amsthm,nicefrac,soul}

\usepackage{mathrsfs}

\newcommand{\e}{\varepsilon}

\numberwithin{equation}{section}
\numberwithin{figure}{section}
\numberwithin{table}{section}

\newtheorem{theorem}[equation]{Theorem}
\newtheorem{lemma}[equation]{Lemma}

\theoremstyle{definition}
\newtheorem{definition}[equation]{Definition}
\newtheorem{remark}[equation]{Remark}
\newtheorem{example}[equation]{Example}

\begin{document}

\begin{frontmatter}
\title{Connectedness in weighted consensus division of graphical cakes between two agents} 
\author[1]{Josef Hanke}
\ead{jph90@cam.ac.uk}

\author[2]{Ana Rita Pires\corref{cor1}}
\ead{apires@ed.ac.uk}

\affiliation[1]{organization={Yusuf Hamied Department of Chemistry, University of Cambridge},
addressline={Lensfield Road},
postcode={CB2 1EW},
city={Cambridge},
country={UK}}

\affiliation[2]{organization={School of Mathematics and Maxwell Institute for Mathematical Sciences, University of Edinburgh},
addressline={JCMB, 
Peter Guthrie Tait Road},
postcode={EH9 3FD},
city={Edinburgh},
country={UK}}

\cortext[cor1]{Corresponding author}

\date{}

\begin{abstract} 
Austin's moving knife procedure was originally introduced to find a consensus division of an interval/circular cake between two agents, each of whom believes that they receive exactly half of the cake.

We generalise this in two ways: we consider cakes modelled by graphs, and let the two agents have unequal, arbitrary entitlements. In this setting, we seek a weighted consensus division -- one where each agent believes they received exactly the share they are entitled to -- which also minimises the number of connected components that each agent receives.

First, we review the weighted consensus division of a circular cake, which gives exactly one connected piece to each agent.
Next, by judiciously mapping a circle to a graph, we produce a weighted consensus division of a star graph cake that gives at most two connected pieces to each agent -- and show that this bound on the number of connected pieces is tight. For a tree, each agent receives at most $h+1$ connected pieces, where $h$ is the minimal height of the tree.  For a connected graphical cake, each agent receives $r+2$ connected pieces, where $r$ is the radius of the graph. Finally, for a graphical cake with $s$ connected components, the division involves at most $s+2r+4$ connected pieces, where $r$ is the maximum radius among all connected components.
\end{abstract}



\end{frontmatter}


\section{Introduction}\label{Intro}
Fair division refers to the problem of allocating resources among individuals with different preferences, according to some notion of fairness. The simplest model is cut-and-choose: to divide a resource (metaphorically, a cake) evenly between two agents, have the first agent cut it into two and the second choose which piece they prefer. The first agent will receive a piece that she views as half (the value of) the cake, but the second agent will most likely receive a piece she values as more than half. 

Austin's moving knife procedure solves this inequity: modelling the cake as an interval (a line segment with two endpoints) the first agent places two knives, one at the left endpoint of the cake and another along the cake such that (in their view) half of the cake lies between the knives. She then moves the two knives continuously to the right, always keeping what she views as half of the cake between the knives until the second agent identifies a point where she agrees that the piece between the knives (and thus the piece(s) outside the knives) is worth half of the cake \cite{Au:SharingACake,BramsTaylor,RoWe:CCA}. 
We often identify the two endpoints of the interval cake to obtain a circular cake -- in this way, the cake is divided into two connected pieces. Each agent agrees on the value of the two pieces (in this case with a value of $\nicefrac{1}{2}$ to both); this is called an \textbf{exact division} or \textbf{consensus division}.
Note that this is stronger than the well-studied \textbf{proportional division} \cite{St:problemofd}, where each agent values their piece at least $\nicefrac{1}{2}$ (more generally, $\nicefrac{1}{n}$ in the case of $n$ agents) \cite{GoLi:Consensus}.
In general, a consensus division splits a resource into $k$ pieces, such that $n$ different agents agree on the valuation of each piece: the valuations of the $k$ pieces need not each be equal to $\nicefrac{1}{k}$, they just need to add up to 1. We focus here on the case of $n=k=2$, and call it \textbf{weighted consensus} to stress that the valuation of the two pieces need not be $\nicefrac{1}{2}$ and $\nicefrac{1}{2}$. Indeed, Austin's moving knife procedure can be modified to produce a consensus division when the two agents have different entitlements (\cite{RoWe:CCA} for rational entitlements and \cite{stack} for arbitrary entitlements). It can be further modified for other situations: for example, Brams, Taylor, and Zwicker give a procedure that produces an envy-free division (not consensus) among four agents \cite{BTZ}, and Shishido and Zeng give one that produces a proportional division of a circular cake among an arbitrary number of agents \cite{SZ:MarkChooseCut}.

In this paper, we tackle the consensus division of more general cakes between two agents with arbitrary entitlements. We consider graphical cakes, i.e., cakes that are modelled by graphs. Interval cakes and circular cakes can be viewed as simple graphical cakes, but we concern ourselves with increasingly complex graphs: first star graphs, then trees, then general graphs. For each setting, we provide a bound on the number of connected components that each agent may receive. These various aspects reflect natural constraints in possible real-world applications. Graphical cakes can model rivers, or road networks or power grids, and it makes sense to try to give each agent a piece that is as connected as possible. For example, fishing rights on a river can be granted to different parties, and some road systems are divided between car rescue companies. Here, a more connected allocation results in a better use of resources. The division of graphical cakes has been studied by various authors. For different types of graphs and numbers of agents,  Bei, Elkind, Segal-Halevi and Suksompong \cite{BS:divgraphcake2019} investigated what is the most proportional allocation possible with the fewest connected components given to each agent can be, whereas Yuen and Suksompong \cite{approx} looked instead at connected allocations with low envy. Consensus divisions may feel fairer to the agents than other usual definitions of fairness (e.g., envy-free, proportional, equitable), and have been studied in various contexts. Simmons and Su \cite{SiSu} studied consensus-halving (i.e. $k=2$ with a 50-50 split, arbitrary $n$) of interval cakes with pieces as connected as possible, Goldberg and Li looked at the weighted version of consensus-halving \cite{GoLi:Consensus}, and Alon and West \cite{Alon,AlonWest} investigated the related problem of necklace splitting (where each agent assigns value only to the beads of a different colour on a section of a necklace).

In Section \ref{sec:austin}, we define the notion of tracing of a graphical cake, which allows us to translate valuations and division between circular cakes and more general graphical cakes. We also make a note on how to implement Austin's moving knife procedure when we allow agents to assign a value of zero to open subsets of the cake --  we did not find this case explicitly addressed elsewhere in the literature. In Section \ref{Circle}, we review the algorithm that produces a weighted consensus division of a circular cake between two agents, which results in each agent receiving one connected piece. A similar results appears in \cite{stack}. The existence of such an algorithm is a direct result of Stromquist-Woodall theorem  \cite[Theorem 1]{SW:StormquistWoodall} for two agents. A repeated application of the Stromquist-Woodall theorem guarantees the existence of a weighted consensus division of a circular cake among $n$ agents that requires at most $(n-1)(2n-2)$ cuts. 

The current paper presents algorithms that produce weighted consensus divisions between two agents, providing a bound on the number of connected components each agent receives, as follows:

\begin{itemize}
\item For a star cake, each agent receives at most two connected pieces (Section \ref{graphdiv}). This bound is sharp.
\item For a tree cake, each agent receives at most $h+1$ connected pieces, where $h$ is the minimal height of the tree (Section \ref{graph}). In this section, we introduce an algorithm for tracing trees (essentially depth-first search) which is crucial in obtaining the bounds in Sections \ref{graph}, \ref{graph2}, and \ref{graph3}.
\item For a connected graphical cake, each agent receives at most $r+2$ connected pieces, where $r$ is the radius of the graph (Section \ref{graph2}).
\item For a general graphical cake with $s$ connected components, the division involves at most $s+2r+4$ connected pieces, where $r$ is the maximum radius among all connected components (Section \ref{graph3}).
\end{itemize}

\noindent The study of weighted consensus division of non-circular cakes between $n\geq 3$ agents remains an open question.

\section{Setup}\label{sec:austin}

Following the definition in \cite[Section 2]{BS:divgraphcake2019}, a \textbf{graphical cake} represented by a finite graph $\Gamma=(V,E)$ consists of a set of $|E|$ intervals, each interval corresponding to an edge of $\Gamma$, with some endpoints of those intervals identified according to the structure of $\Gamma$. By abuse of notation, we often denote the cake also by $\Gamma$.

In this framework, an interval cake is a graphical cake represented by a graph consisting of a single edge (with the two endpoints not identified with each other), a circular cake is a graphical cake represented by a graph consisting of a single edge whose two endpoints are identified with each other, and a multi-cake (i.e., a cake consisting of a finite number $m$ of disjoint intervals) is a graphical cake represented by a graph consisting of $m$ edges, with no identification of any endpoints. A \textbf{piece} or subset of the cake is a finite union of subintervals from one or more edges in $E$. We say that a piece of cake is connected if it is path-connected under the identifications of the endpoints of edges mentioned above.

\begin{definition}
A \textbf{tracing of a graphical cake} $\Gamma=(V,E)$ is a surjective function $g:[0,1]\to\Gamma$ such that $g^{-1}(V)$ is a finite set of points and the restriction  $$g|_{[0,1]\setminus g^{-1}(V)}:[0,1]\setminus g^{-1}(V)\to \Gamma\setminus V$$
is bijective.
\end{definition}

A \textbf{valuation} $v$ on a continuous\footnote{Continuous as opposed to discrete.} set $X$ (such as a graphical cake) is a non-atomic probability measure on $X$. In particular, this means that each point in $X$ has zero value, and so for any subset $Y\subset X$ and any finite subset $S\subset Y \subset X$ we have 
\begin{equation}\label{eq:points}
v(Y\setminus S )= v(Y).
\end{equation}

\begin{remark}\label{nonatomic}
A valuation on a graphical cake $\Gamma$ induces, via a tracing $g$ of $\Gamma$, a valuation on $[0,1]\setminus g^{-1}(V)$. Therefore, by the above, it induces also a valuation on the circular cake $[0,1]/(0\sim 1)$.  
\end{remark}

We denote the valuation functions of the two agents $A$ and $B$ on a graphical cake $\Gamma$ by $v_A$ and $v_B$, respectively. By \eqref{eq:points}, an isolated point (such as a vertex) can be assigned to both agents, and so it makes sense to define a \textbf{division of the cake} $\Gamma$ between the two agents as a pair of pieces of cake $(\alpha,\beta)$ such that $\alpha \cup\beta=\Gamma$ and $\alpha \cap\beta$ consists of finitely many points.

\begin{definition}
\label{weightedproportional}
Let $0\leq t\leq 1$ and $(t,1-t)$ be the vector of entitlements of agents $A$ and $B$. 
A division $(\alpha, \beta)$ is a \textbf{weighted consensus division} if both agents believe that they (and therefore the other agent) have been allocated exactly their entitlement: $$v_A(\alpha)=v_B(\alpha)=t\ \text{ and }\ v_A(\beta)=v_B(\beta)=1-t.$$
\end{definition}

\begin{remark}\label{inducedivision}
Given a tracing $g:[0,1]\to\Gamma$ of a graphical cake $\Gamma$, a division of the circular cake $[0,1]/(0\sim 1)$ can be mapped via the bijection $g|_{[0,1]\setminus g^{-1}(V)}$ to a division of $\Gamma\setminus V$. For valuation purposes, we can assign the vertices in $V$ to whichever agents we want, to get a division of the entire graphical cake $\Gamma$. But for connectivity purposes, our convention is that a vertex should be assigned precisely to the agents that own the terminal segment of an edge incident to that vertex.
By construction, if the division of the circular cake is weighted consensus, then so is the division induced on $\Gamma$.
\end{remark}

\subsection{A remark on Austin's moving knife procedure}\label{rmkmoving}

Note that a valuation function can assign a value of zero to open subsets of the cake. This requires a more subtle use of the Intermediate Value Theorem in the argument of Austin's moving knife procedure than otherwise, and we detail that here because we did not find this case explicitly addressed elsewhere in the literature.

Consider an interval cake $[0,1]$ and an entitlement vector $(\nicefrac{1}{2},\nicefrac{1}{2})$; we follow Austin's moving knife procedure described in \cite[p.214-215]{Au:SharingACake} and summarised above in Section \ref{Intro}. We let $c$ denote the position of the first knife and $\kappa(c)$ the position of the second knife, positioned by agent $A$ such that $v_A\bigl([c,\kappa(c)]\bigr) = \nicefrac{1}{2}.$

Implicit in the definition of the $v_P$ valuation function as a probability measure are the probability density and cumulative distribution functions $f_P$ and $F_P$, respectively. If we do not allow a value of zero to be assigned to open subsets of the cake, $F_P$ is strictly monotone increasing and thus has a well-defined continuous inverse. In this case, the position of the second knife can indeed be viewed as a continuous function of the position of the first:
\begin{equation}\label{eq:kappa}
\kappa(c) = F_A^{-1}\bigl(\frac12+F_A(c)\bigr).
\end{equation}
The Intermediate Value Theorem is then applied to the following continuous function:
\begin{equation}\label{eq:IMV}
c\mapsto v_B\bigl([c, \kappa(c)]\bigr) = F_B\bigl(\kappa(c)\bigr) - F_B\bigl(c\bigr).
\end{equation}

\begin{figure}[h!]
\begin{center}
\includegraphics[width = 0.8\textwidth]{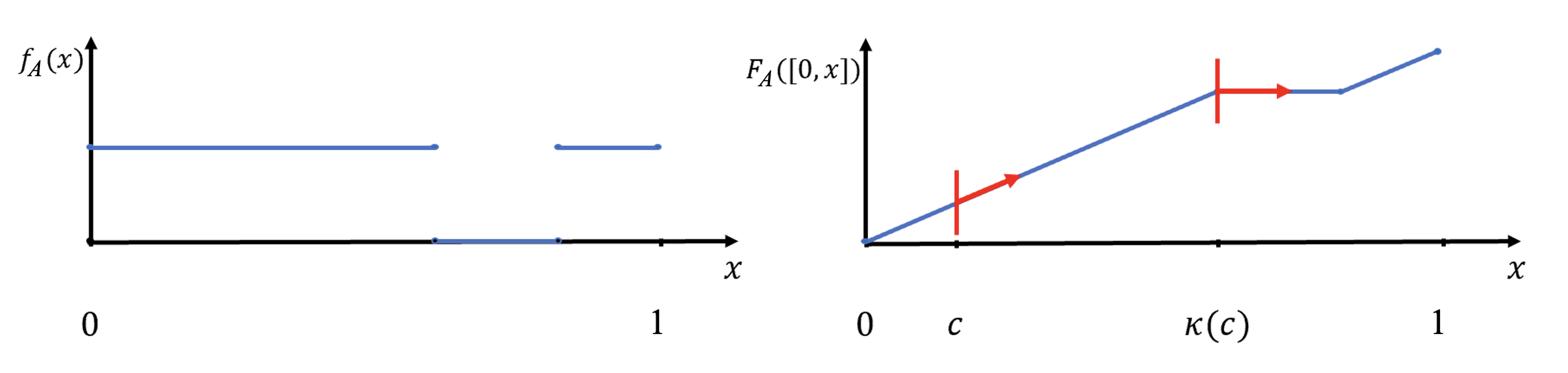}
\caption{If the probability density function of agent $A$ has a zero-valued interval as in the example on the graph on the left-hand side, then the cumulative distribution function will be constant on the same interval as on the graph on the right-hand side. In that case, the function $\kappa$ in \eqref{eq:kappa} is not well defined. Indeed, at the point in the procedure where the knives are at the locations indicated in red on the graph on the right, keeping $c$ constant and increasing slightly the value denoted in the figure by $\kappa(c)$ would still yield $v_A\bigl([c,\kappa(c)]\bigr)=\nicefrac12$.}
\label{zero_issue}
\end{center}
\end{figure}

Figure \ref{zero_issue} illustrates the problem with allowing $f_P$ to vanish on an interval. In practical terms, in many situations it may be perfectly fine to edit the agents' probability density functions slightly, assigning a small positive value to any vanishing intervals and rescaling the rest of the function accordingly. Alternatively, it is possible to modify the moving knife procedure in order to ensure that a continuous function $\kappa(c)$ can be defined and the Intermediate Value Theorem can be applied to the function \eqref{eq:IMV}. The modification is as follows: 

Move the knives as in Austin's moving knife procedure until the rightmost knife reaches a point $\kappa(c)$ where 
$f_A\bigl(\kappa(c)\bigr) =0.$
Stop moving the left knife, and continue moving the right knife until $f_A\bigl(\kappa(c)\bigr) >0$. Start moving both knives again and repeat this procedure until the whole cake is covered.

\section{Circular cakes}
\label{Circle}

The Stromquist-Woodall theorem \cite[Theorem 1]{SW:StormquistWoodall} in the case of $n=2$ agents guarantees the existence of a subset of a circular cake that both agents value as $t$, for any given $0\leq t\leq 1$, requiring only $2n-2=2$ cuts. This implies that the agents also agree on the value of the remainder of the cake, and therefore guarantees the existence of a weighted consensus division with entitlements $t$ and $1-t$, in which both agents receive a connected piece. The existence of such a division suffices for the proofs of the results in Sections \ref{graphdiv}--\ref{graph2}, but for completeness we present here a simpler proof for our particular case of two agents. A similar proof appears on \cite{stack} as noted by an anonymous referee.

\begin{theorem}  
\label{thm2}
Let two agents be owed arbitrary entitlements of a circular cake. There exists a weighted consensus division of the cake such that each agent's piece is connected.
\end{theorem}

\begin{proof}
Let the entitlement vector be $(t,1-t)$ and denote the circular cake by $\Theta$. When writing an interval $[a, b]\in\Theta$, we mean as usual the interval $[a,b]\subset[0,1]$ if $a\leq b$ and $[a,1]\cup[0,b]$ if $a>b$. This implies $$v_P\left([a,b]\right) = \left\{
	\begin{array}{ll}
		\int_a^b f_P(\theta)\ d\theta  & \mbox{if } a \leq b \\
		1-\int_b^a f_P(\theta)\ d\theta & \mbox{if } a > b
	\end{array}
\right. $$
for each agent $P\in\{A,B\}$.

Agent $A$ places two knives on the cake; the first at $c=0$ and the second at $\kappa(0)$ such that she values the piece of cake between the knives as her entitlement $t$. Then the first knife is moved around the cake anticlockwise, and agent $A$ moves the second knife such that she consistently values the cake between the two knives as equal to $t$. This defines the position of the second knife as a function of the position of the first knife: $\kappa: \Theta \rightarrow \Theta$ such that $v_A\bigl([c,\kappa(c)]\bigr)=t$. As noted in Section \ref{rmkmoving}, if $f_A$ is strictly positive, then $\kappa$ is well-defined; otherwise, we can modify how to move the second knife; this is explained in the last paragraph of Section \ref{rmkmoving}.

We claim that at some point, agent $B$ also values the cake between the knives as $t$. At this point, the division $\bigl([c,\kappa(c)],[\kappa(c),c]\bigr)$ is weighted consensus. For the purpose of a proof by contradiction, assume no such point exists. By the continuity of $c\mapsto v_B\bigl([c,\kappa(c)]\bigr)$, we can assume that for all $c\in\Theta$ we have  
$v_B\bigl([c,\kappa(c)]\bigr)>t$. By compactness of $\Theta$, this means that there exists an $\e>0$ such that $v_B\bigl([c,\kappa(c)]\bigr)\geq t+\e.$

We lift the probability density functions $f_P:\Theta\to\mathbb{R}$ to $\widetilde{f_P}:\mathbb{R} \rightarrow \mathbb{R}$ by defining $\widetilde{f_P}(x+1) = \widetilde{f_P}(x)$ and obtain corresponding cumulative distribution functions $\widetilde{F_P}:\mathbb{R}\to \mathbb{R}$ which satisfy $$\widetilde{F_P}(x+1) = \int_0^{x+1}\widetilde{f_P}(\theta)\ d\theta = \int_0^x\widetilde{f_P}(\theta)\ d\theta + \int_x^{x+1}\widetilde{f_P}(\theta)\ d\theta = \widetilde{F_P}(x)+1.$$ 
As a consequence, the difference $\bigl(\widetilde{F_A} - \widetilde{F_B}\bigr): \mathbb{R} \rightarrow \mathbb{R}$ is periodic with period 1
 
and thus descends to a function $\bigl(F_A - F_B\bigr):\Theta \rightarrow \mathbb{R}.$ By compactness of $\Theta$ there exists an upper bound $M$ such that $|\bigl(F_A-F_B\bigr)| \leq M$ on $\Theta$. 

The function $\kappa:\Theta\to \Theta$ also induces a function $\widetilde{\kappa}:\mathbb{R}\to\mathbb{R}$ in the obvious way. We use it to inductively define a sequence $(\widetilde{c_i})_{i\in\mathbb{N}}\subset \mathbb{R}$ via $\widetilde{c_0} = 0$ and $ \widetilde{c_{i+1}} = \widetilde{\kappa}(\widetilde{c_i})$, which then induces a corresponding sequence $(c_i)_{i\in\mathbb{N}}\subset \Theta$.

Note that by construction $\widetilde{F_A}(\widetilde{c_{i+1}}) - \widetilde{F_A}(\widetilde{c_i}) = t$ and $\widetilde{F_B}(\widetilde{c_{i+1}}) -\widetilde{F_B}(\widetilde{c_i}) = v_B\bigl([c_i,c_{i+1}]\bigr)\geq t+\e$. Now
$$\bigl(F_A - F_B\bigr)(c_{i+1}) = \widetilde{F_A}(\widetilde{c_{i+1}}) - \widetilde{F_B}(\widetilde{c_{i+1}}) \leq \widetilde{F_A}(\widetilde{c_{i}}) - \widetilde{F_B}(\widetilde{c_{i}}) - \e = \bigl(F_A -F_B\bigr)(c_{i}) - \e$$
which implies
$\bigl(F_A - F_B\bigr)(c_i)  \leq \bigl(F_A - F_B\bigr)(0) -i\e$, contradicting the bound $|\bigl(F_A - F_B\bigr)|\leq M$.

\end{proof}

\section{Star cakes}\label{graphdiv}

We now turn to the case of a star-shaped graphical cake. An ($l$-)\textbf{star graph} is a tree on $l+1$ vertices with one vertex having degree $l$ and the other $l$ vertices having degree one. We assume that $l$ is at least three since, for cake-cutting purposes, $1$-stars and $2$-stars can be modelled as interval cakes, in which case one agent receives one connected piece and the other gets at most two (by applying Theorem \ref{thm2} without identifying the endpoints of $[0,1]$).

\begin{theorem}\label{bigstar}
Let two agents be owed arbitrary entitlements of a star graph cake. There exists a weighted consensus division of the cake such that each agent receives at most two connected pieces.
\end{theorem}

\begin{proof}
Denote the root of the star graph $\Gamma$ by $R$, the edges by $e_1, e_2, \ldots, e_{l}$, and the corresponding degree-one vertices by $w_1, w_2,\ldots, w_{l}$. Let $g$ be a tracing of $\Gamma$ that starts at $R$, follows the edge $e_1$ to its endpoint $w_1$, then jumps to $R$ again and repeats for each edge $e_2, e_3, \ldots$ in turn, until it terminates at $w_{l}$.

As noted in Remark \ref{nonatomic}, the agents' valuation functions on $\Gamma$ induce valuation functions on a circular cake. Theorem \ref{thm2} gives a weighted consensus division of a circular cake where each agent's piece is connected. By Remark \ref{inducedivision} this then induces a weighted consensus division of $\Gamma$.

\begin{figure}[h!]
\begin{center}
\includegraphics[width=0.8\textwidth]{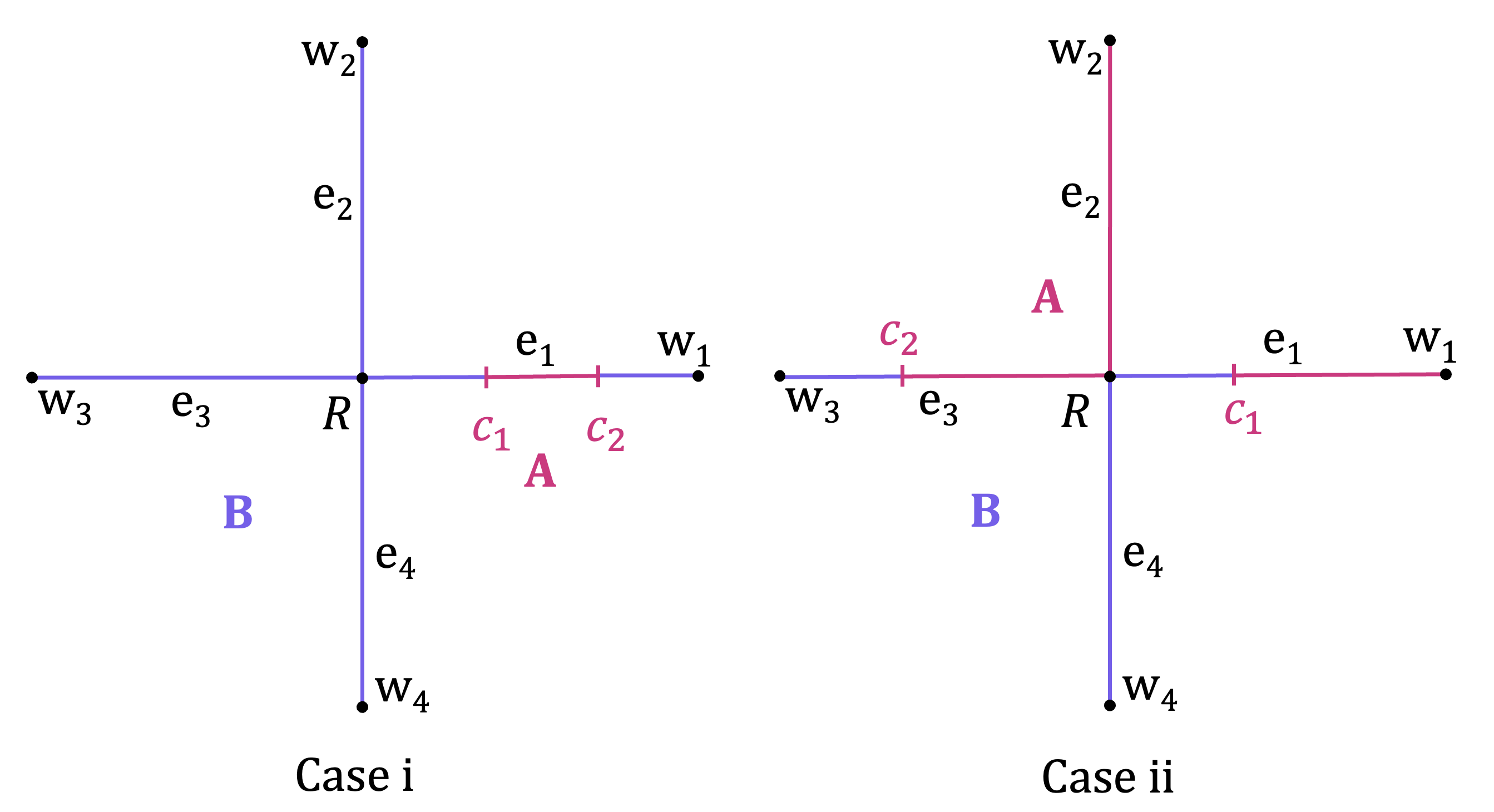}
\caption{Case (i) with $c_1$ and $c_2$ on the same branch. Case (ii) with $c_1$ and $c_2$ on different branches. The pink piece is assigned to agent $A$ and the purple piece to agent $B$.}
\label{circle2}
\end{center}
\end{figure}

It remains to show that this division of $\Gamma$ gives each agent at most two connected pieces. Let $c$ and $\kappa(c)$ denote the two cuts made on the circular cake and let $c_1=g(c)$ and $c_2=g(\kappa(c))$.

There are two distinct cases to consider: either (i) the two cuts, $c_1$ and $c_2$, are on the same edge $e_i$, or (ii) they are on different edges, $e_i$ and $e_j$; see Figure \ref{circle2}.

\noindent\underline{Case (i):}
In this case, one agent receives the single connected piece $[c_1, c_2]\subset e_i$, while the other receives two pieces: the outer remainder of the edge $e_i$, which is $[c_2, w_i]$, plus the rest of the graph (including $[R, c_1]$, the inner remainder of the edge $e_i$) that is connected through $R$.

\noindent\underline{Case (ii):} Without loss of generality, let $i<j$.  Agent $A$ receives the piece $[c_1,w_i]$ as well as the following piece, connected through $R$:
$$[R,c_2]\cup \bigcup_{k=i+1}^{j-1}e_k.$$ Agent $B$ receives the connected piece $[c_2,w_j]$ as well as the following piece, connected through $R$:
$$[R,c_1]\cup \bigcup_{k=j+1}^{l}e_k\bigcup_{k=1}^{i-1}e_k.$$ 
\end{proof}

To conclude this section, we show with Example \ref{example} that Theorem \ref{bigstar} is sharp in the sense that it is not generally possible to guarantee a consensus division such that one of the agents gets at most one piece while the other one gets at most two pieces. Whether this bound is sharp for a weighted proportional division of a star graph cake (recall that a weighted consensus division is also a weighted proportional division) remains an open question.

\begin{figure}[h!]
\begin{centering}
\includegraphics[width=0.4\textwidth]{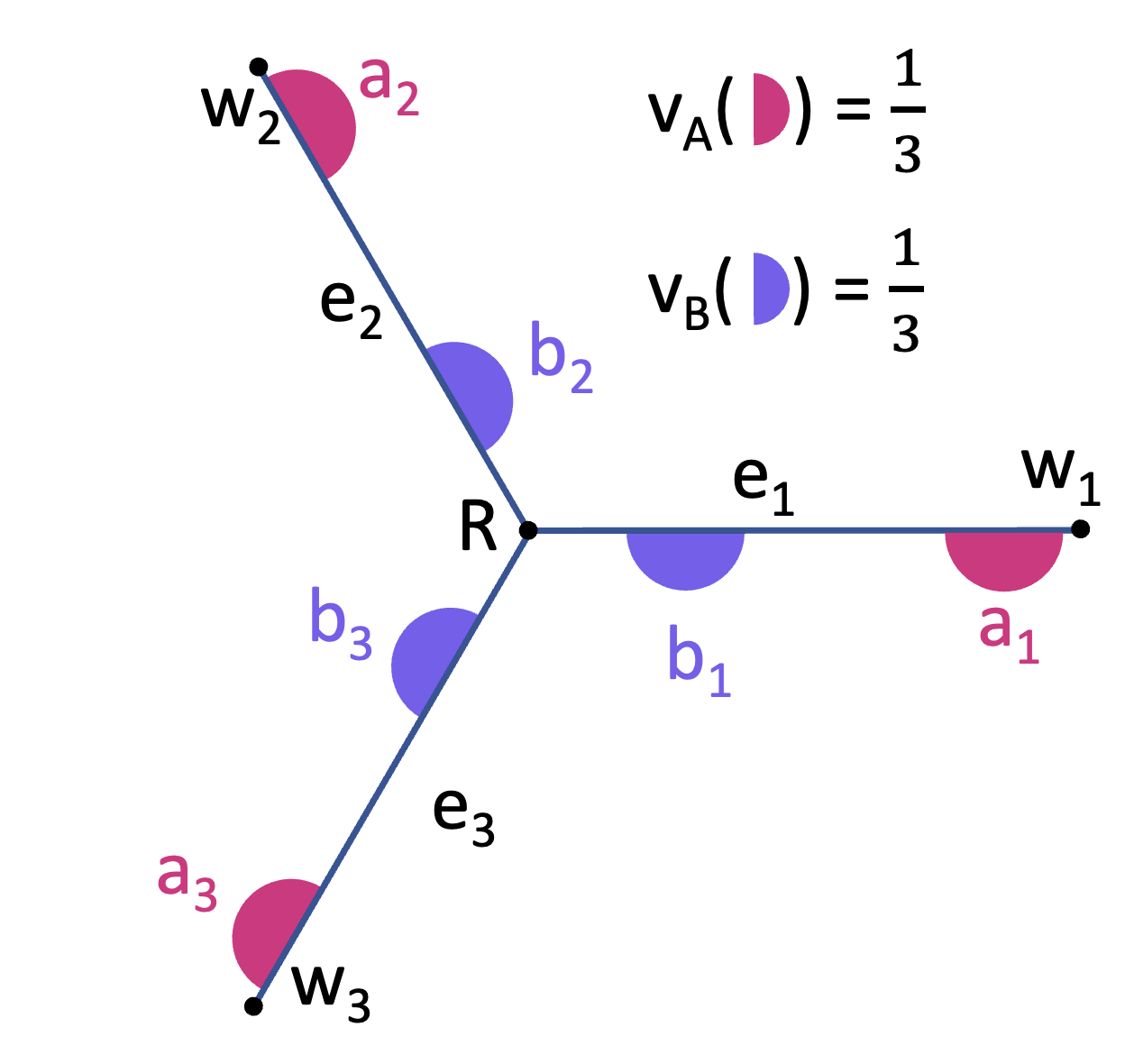}
\caption{An example of a star graph cake where there is no consensus division where one agent receives a single connected piece while the other receives at most two.}
\label{counterexamplefig}
\end{centering}
\end{figure}

\begin{example}\label{example}
Let $l=3$ and $t=\nicefrac12$, and consider the value functions represented in Figure \ref{counterexamplefig}. The corresponding distribution functions must vanish at various points, but in practice, we can modify the functions slightly to avoid the type of complications described in Section \ref{rmkmoving}. 
Assume, towards a contradiction, that there exists a consensus division $(\alpha,\beta)$ of $\Gamma$ such that one of the agents receives one connected piece and the other receives two. 
Agent $A$ must receive a piece touching at least two of $\{a_1,a_2,a_3\}$, otherwise the value of her piece is at most $\nicefrac{1}{3}$. Agent $B$ must also receive a piece touching at least two of $\{a_1,a_2,a_3\}$, otherwise agent $A$'s value is at least $\nicefrac{2}{3}$.
Thus, as we assumed one of the agents $A$ or $B$ receives a connected piece, one of them must receive at least two of $\{b_1,b_2,b_3\} $ entirely. If it is agent $B$, then agent $B$ values her piece at least $\nicefrac{2}{3}$; if it is agent $A$, then agent $B$'s valuation is at most $\nicefrac{1}{3}$. In both cases, the allocation is not a consensus allocation.

Note that we can find a proportional division of the same star graph cake by letting $\alpha = a_1 \cup a_2$ (two connected pieces) and $\beta = \Gamma \setminus \alpha$ (one connected piece). In this case, $v_A(\alpha)= \nicefrac{2}{3} \geq \nicefrac{1}{2}$ and $v_B(\beta) = 1 \geq \nicefrac{1}{2}$.
\end{example}
 
\section{Tree cakes}\label{graph}

Next, we look at the case when the graphical cake is a tree -- a connected graph with no cycles. We extend the tracing of star graphs described in the proof of Theorem \ref{bigstar} to a tracing of general tree graphs, described in Algorithm \ref{algorithm}.

Let $\Gamma$ be a tree graph and choose its root $R$ such that $\Gamma$ has minimal height, $h$. Label the subtrees below $R$ as $\gamma_1,...\gamma_{l}$, each with a \textbf{rooted edge} $E_k$ connecting $\gamma_k$ to $R$; see Figure \ref{example_2proof}. Note that the height of each of these subtrees is at most $h-1$.  A \textbf{leaf edge} is an edge incident on a leaf vertex (i.e.\ a vertex of degree one). Our convention is that the subtree $\gamma_k$ corresponding to a rooted leaf edge $E_k$ is empty.  

The tree tracing $\texttt{TT}(\Gamma,R)$ essentially implements a depth-first search and is defined inductively using the following algorithm:

\begin{algorithm}
\caption{\texttt{TT}(G, x)}\label{algorithm}
\begin{algorithmic}[1]
\WHILE{there exists an untraced edge incident to $x$}
    \STATE Let $e$ be the leftmost of these edges.
    \STATE Let $W$ be the other endpoint of $e$.
    \STATE Let $G'$ be the subtree rooted at $W$.
    \STATE Trace $e$ from $x$ to $W$.
    \STATE \texttt{TT}($G'$, $W$).

\ENDWHILE
\end{algorithmic}
\label{ttt}
\end{algorithm}

\vspace{1em}

Note that $\texttt{TT}(\Gamma,R)$  covers each subtree $\gamma_i$ independently and completely, directly after tracing along $E_i$. This allows for results based on inductive reasoning, using Algorithm \ref{ttt}, to be found.

\begin{lemma}
\label{one_cut}
Pausing $\texttt{TT}(\Gamma, R)$ at any point divides a tree $\Gamma$ with root $R$ and height $h\geq 1$ into two sections: the section already traced has one connected piece containing $R$, and the section still untraced has at most $h+1$ connected pieces. 
\end{lemma}

\begin{proof}
    We proceed by complete induction on the height of the tree $\Gamma$ with root $R$. 

    \textbf{Base case:} For trees of height $1$, i.e.\ star graphs, each edge is traced from the root $R$ to its leaf vertex. Hence, the traced section is connected through $R$. The other section has at most two connected components: at most one containing $R$, and at most one not containing $R$.

    \textbf{Induction hypothesis:} Assume that for any tree of height $n<h$, pausing at a point along its tree tracing \texttt{TT} divides it into two sections: the section already traced, which has one connected piece containing $R$, and the section still untraced has at most $n+1$ connected pieces. 
    
    \textbf{Inductive step:} Let $\Gamma$ be a rooted tree of height $h\geq 2$. 
    We consider two separate cases that depend on the location of the pause along the tree tracing \texttt{TT}$(\Gamma,R)$: (i) on a rooted edge $E_i$, and (ii) on a subtree $\gamma_i$ of $\Gamma$.
    This corresponds to pausing the $\texttt{TT}$ algorithm during (i) step 5 of $\texttt{TT}(\Gamma, R)$ or (ii) step 6, i.e. during $\texttt{TT}(\gamma_i,\text{root of } \gamma_i)$. See Figure \ref{example_2proof} for an illustration of these cases (ii).

\begin{figure}[h!]
\begin{centering}
\includegraphics[width=0.8\textwidth]{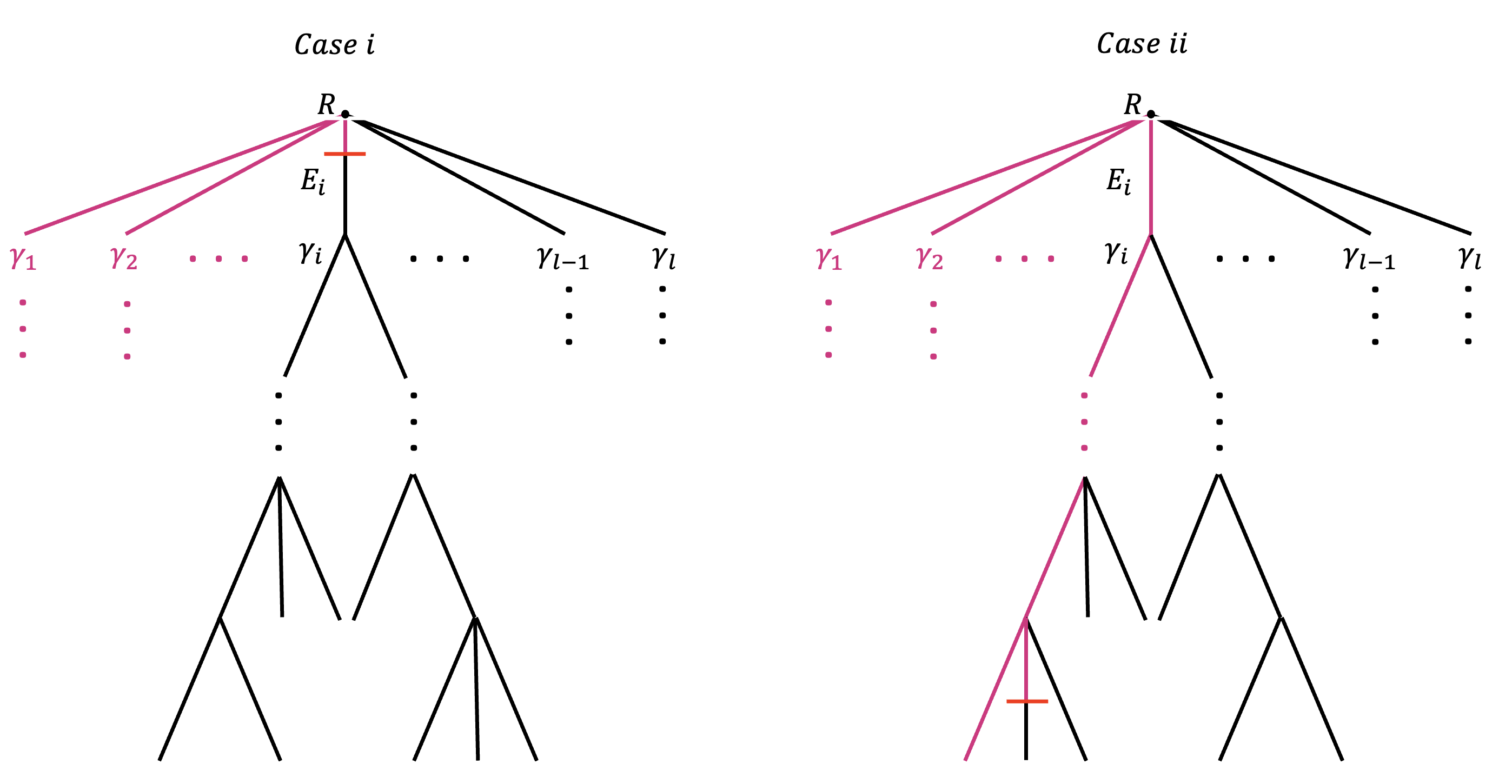}
\caption{Cases (i) and  (ii) of the possible locations of a pause in \texttt{TT}($\Gamma, R$), as indicated by the red lines. The traced pieces are in pink and the untraced pieces are in black.}
\label{example_2proof}
\end{centering}
\end{figure}

\noindent\underline{Case (i):} Recall that the tree tracing algorithm traces each rooted edge before tracing the corresponding whole subtree and moving on to the next rooted edge.
Therefore, the untraced section consists of at most two connected pieces (and in particular, at most $h+1$ connected pieces, as we wanted to prove). The first is the untraced part of $E_i$ along with $\gamma_i$. If $i<l$ then the untraced section contains a second connected piece: $\bigcup_{k=i+1}^{l}(\gamma_{k}\cup E_k)$, which is connected through $R$. The traced section of the tree consists of a single connected piece that contains $R$.

\noindent\underline{Case (ii):} Here we apply the induction hypothesis to the subtree $\gamma_i$, which has height at most $h-1$: by the induction hypothesis, we know that the traced part of $\gamma_i$ has one connected piece and its untraced part has at most $h$ connected pieces. Now, the untraced part of $\Gamma$ contains the untraced portion of $\gamma_i$. If $i<l$, it also contains $\bigcup_{k=i+1}^{l}(\gamma_{k}\cup E_k)$ (which is connected through $R$). Therefore, the untraced portion of $\Gamma$ contains at most $h+1$ connected pieces. On the other hand, the traced part of $\Gamma$ consists of $\bigcup_{k=1}^{i-1}(\gamma_k\cup E_k)$ (which is connected through $R$) together with the traced portion of $\gamma_i$ connected to the previous piece through the traced $E_i$ (which contains both $R$ and the root of $\gamma_i$).

This concludes the inductive step, and hence the proof of the Lemma.
\end{proof}

\begin{remark}
We can apply the cut-and-choose method to the tree tracing described above to obtain a proportional division of a tree cake between two agents that have an entitlement vector of $(\nicefrac12,\nicefrac12)$ in the following way: let agent $A$ move a knife along the tree cake $\Gamma$ according to $\mathtt{TT}(\Gamma,R)$ and cut it whenever they think that both the untraced and traced portions are worth $\nicefrac{1}{2}$ of the cake. Then, agent $B$ chooses which half they want, thus yielding a proportional allocation. By Lemma \ref{one_cut}, one agent receives a single connected piece, and the other receives at most $h+1$ connected pieces. This is the same number of connected pieces that appear in the proportional division in Theorem 4.22 of \cite{BS:divgraphcake2019}. That theorem is a corollary of Lemma 4.21 in the same paper, which is very similar to our Lemma \ref{one_cut}, just with a different choice of tracing of the tree.
\end{remark}

Next, we use Lemma \ref{one_cut} to obtain a weighted consensus division.

\begin{theorem}\label{tree}
Let two agents be owed arbitrary entitlements of a tree graph of minimal height $h$. There exists a weighted consensus division of the cake such that each agent receives at most $h+1$ connected pieces.
\end{theorem}

\begin{proof} 
Let $\Gamma$ be a tree with root $R$ and minimal height $h$, and consider the tracing $\texttt{TT}(\Gamma,R)$. As noted in Remark \ref{nonatomic}, the agents' valuation functions on $\Gamma$ induce, via this tracing, valuation functions on a circular cake. 
Theorem \ref{thm2} cuts a circular cake at two points to produce a weighted consensus division where each agent's piece is connected. By Remark \ref{inducedivision}, this then induces two cuts on $\Gamma$ and a corresponding weighted consensus division.
We claim that this division of $\Gamma$ results in at most $h+1$ connected pieces for each agent. We will prove this by complete induction on $h$.

 \textbf{Base case:} The tracing \texttt{TT} applied to a tree of height 1, i.e. a star graph, is exactly the tracing described in the proof of Theorem \ref{bigstar}. The rest of that proof shows that, in this case, each agent receives at most two connected pieces, as desired.

 \textbf{Induction hypothesis:} Assume that for any tree of height $n<h$, the consensus division obtained as described in the first paragraph of this proof assigns at most $n+1$ connected pieces to each agent.

 \textbf{Inductive step:} Let $\Gamma$ be a tree of height $h\geq 2$ with root $R$. The process described in the first paragraph of this proof yields two cuts on $\Gamma$ which determine its consensus division. There are two cases of where these cuts could lie on $\Gamma$, relative to each other: either 
(i) the two cuts are on the same subtree, $\gamma_i\cup E_i$ or 
(ii) the two cuts are on different subtrees, $\gamma_i \cup E_i$ and $\gamma_j \cup E_j$.

\underline{Case (i):}
If both of the cuts are on $E_i$ one agent gets a single connected piece (the segment of $E_i$ between the cuts), and the other gets two (one piece consists of $\gamma_i$ together with the segment of $E_i$ adjacent to it, the other consists of all other subtrees and respective rooted edges, together with the segment of $E_i$ connected to the root).

If the first cut is on $E_i$ and the second on $\gamma_i$ (which is a tree of height at most $h-1$), then one agent
gets a single connected piece (the rest of $E_i$ from the cut to the root of $\gamma_i$ as well as the part of $\gamma_i$ traced until the second cut, connected through the root of $\gamma_i$). The other agent gets the part of  $\gamma_i$ still untraced when the second cut is made (consisting of at most $h$ connected pieces by Lemma \ref{one_cut}) as well as the connected piece $\bigcup_{k\neq i}(\gamma_k\cup E_k)$ with the segment of $E_i$ between $R$ and the first cut -- this comes to $h+1$ connected pieces. Therefore, each agent receives at most $h+1$ connected pieces.

If both cuts are on $\gamma_i$ (which is a tree of height at most $h-1$), then by the induction hypothesis, $\gamma_i$ can be assigned to the agents such that each agent receives at most $h$ connected pieces. All of the other subtrees, along with $E_i$, are assigned to the same agent. As all the subtrees are connected through $R$, the agent assigned $E_i$ receives at most $h+1$ connected pieces, and the other agent receives at most $h$ connected pieces. 

\underline{Case (ii):} 
In this case, the end of the rooted branches $E_i$ and $E_j$ connected to $R$ are assigned to different agents, let us say $A$ and $B$ respectively.  Let us consider $\gamma_i \cup E_i$. If the cut on this part is on $E_i$, then both $A$ and $B$ receive one connected piece of $\gamma_i \cup E_i$. If the cut is on $\gamma_i$, then by Lemma \ref{one_cut}, $\gamma_i$ is divided such that agent $A$ gets one connected piece, and agent $B$ gets at most $h$ connected pieces of $\gamma_i\cup E_i$. 

For $\gamma_j\cup E_j$, this is the opposite: agent $B$ gets one connected piece and agent $A$ gets at most $h$ connected pieces of $\gamma_j \cup E_j$. As all the $\gamma_k$'s are connected through $R$, and each of the other subtrees is assigned wholly to one agent, overall, each agent gets at most $h+1$ connected pieces.
\end{proof}

\begin{remark}
Example \ref{example} shows that Theorem \ref{tree} is sharp for trees of height $h=1$. Whether it is sharp for trees of other heights remains an open question: is there a tree of height $h>1$ such that any weighted consensus division requires at least $2h+2$ connected pieces to be split among the two agents?
\end{remark}

\section{Connected graphical cakes}\label{graph2}

In this section, we tackle the case of a connected graphical cake $G$. Recall that a \textbf{spanning tree}  of a connected graph $G$ is a subgraph of $G$ that is a tree (a connected graph without cycles) and includes all of the vertices of $G$. It can be obtained by inductively deleting edges of $G$ that form cycles in $G$. We define a related notion, as follows:

\begin{definition}
Given a connected graph $G$, an  \textbf{edge-spanning tree} of $G$ is any tree obtained by inductively detaching edges that form cycles in $G$ from one of their endpoints and adding a vertex to serve as a new endpoint. 
\end{definition}

The result is a tree that contains as many edges as $G$ had. Equivalently, an edge-spanning tree of $G$ can be constructed (not uniquely) from a given spanning tree of $G$ by augmenting it: for each edge $e$ of $G$ that was removed when constructing the spanning tree, add it back as a new leaf edge incident to either endpoint of $e$. Doing this for all removed edges increases the height of the original spanning tree by at most one, because new leaf edges are only appended onto vertices of the spanning tree (and not to the new endpoint of another new leaf edge).

The radius of a graph is the minimum, over all vertices $v$, of the maximum distance from $v$ to any other vertex of the graph. Therefore, if a graph $G$ has radius $r$, then a spanning tree of $G$ of minimal height has height at most $r$, and an edge-spanning tree of $G$ of minimal height has height at most $r+1$.

\begin{theorem}
Let two agents be owed arbitrary entitlements of a connected graphical cake of radius $r$. There exists a weighted consensus division of the cake such that each agent receives at most $r+2$ connected pieces.
\label{connected_graph}
\end{theorem}

\begin{proof}
Let $G$ be a connected graph and $\Gamma$ be a minimal height edge-spanning tree of $G$; its height is at most $r+1$. The agents' valuation functions on $G$ induce valuation functions on $\Gamma$.
Theorem \ref{tree} guarantees that there exists a weighted consensus division of $\Gamma$ such that each agent receives at most $r+2$ connected pieces. 
We obtain $G$ again by identifying some of the vertices of $\Gamma$ -- those that were detached to break the cycles of $G$ in the construction of $\Gamma$ -- and giving each agent the pieces of $G$ that correspond to the pieces of $\Gamma$ they were assigned. 

 Whenever we identify two points assigned to the same agent, if those two points are in the same connected piece of $\Gamma$ that is being assigned to that agent, then the number of pieces assigned to the agent remains unchanged. If those two points are in different connected pieces of $\Gamma$ being assigned to that agent, then doing this decreases the number of connected pieces by one. When we identify points assigned to different agents, this does not change the number of connected pieces that each agent receives. Hence, this process cannot increase the number of connected pieces that each agent receives, and the result is proven.
\end{proof}

\begin{remark}
Following the tree tracing $\mathtt{TT}$ of an edge-spanning tree of a (connected) graph $G$ as implied in the proof of Theorem \ref{connected_graph} induces a tracing of $G$. Another tracing of a graph is presented in \cite[Definition 4.3]{BS:divgraphcake2019}. The authors apply cut-and-choose to that tracing, for example, to construct a proportional allocation between two agents with an entitlement vector of $(\nicefrac12,\nicefrac12)$ such that each agent receives one connected piece, in the case when the graph $G$ is almost-bridgeless; see  \cite[Theorem 4.5]{BS:divgraphcake2019}.
\end{remark}

\section{General graphical cakes}\label{graph3}

Finally, we look at the case of general graphical cakes, i.e., we drop the hypothesis of connectedness in Theorem \ref{connected_graph}. Here, a graph $\mathcal{G}$ is a disjoint union of $s$ connected components $G_k$, each with radius $r_k$:
$$\mathcal{G}=\bigcup_{k=1}^{s}G_k.$$
We denote by $r=\max_{k}r_k$ the maximal radius over all components.

An \textbf{edge-spanning forest} $\Psi$ of $\mathcal{G}$ is the disjoint union of edge-spanning trees $\Gamma_k$ of each of the connected components $G_k$:
$$\Psi=\bigcup_{k=1}^{s}\Gamma_k.$$
We say that an edge-spanning forest has minimal height when each edge-spanning tree is itself of minimal height.

\begin{theorem}
Let two agents be owed arbitrary entitlements of a graphical cake composed of $s$ connected components with maximal radius $r$. There exists a weighted consensus division of the cake involving at most $s+2r+4$ connected pieces. More precisely,
\begin{enumerate}[(i)]
\item either $s-1$ of the $s$ connected components are each assigned wholly to one of the two agents, and the remaining connected component is divided such that each agent receives at most $r+2$ connected pieces from that component,
\item or $s-2$ of the $s$ connected components are each assigned wholly to one of the two agents, and the remaining two connected components are divided such that each agent receives at most $r+3$ connected pieces coming from those components.
\end{enumerate}
\label{disconnected_graph}
\end{theorem}

\begin{proof}
    Let $\mathcal{G}=\bigcup_{k=1}^{s}G_k$ be a graph with $s$ connected components, with maximal radius $r$. Let $\Psi=\bigcup_{k=1}^{s}\Gamma_k.$ be a minimal height edge-spanning forest of $G$, where each $\Gamma_k$ is an edge-spanning tree of the connected component $G_k$ and has root $R_k$.
    The height of each $\Gamma_k$ is at most $r+1$. The agents' valuation functions on $G$ induce valuation functions on $\Psi$.

    Trace the forest $\Psi$ by concatenating $\texttt{TT}(\Gamma_k,R_k)$ for $k=1, 2, \ldots, s$. 
    As noted in Remark \ref{nonatomic}, the agents' valuation functions on $\Psi$ induce, via this tracing, valuation functions on a circular cake. Theorem \ref{thm2} cuts a circular cake at two points, $c$ and $\kappa(c)$, to produce a weighted consensus division where each agent's piece is connected: agent $A$ receives the piece between $c$ and $\kappa(c)$ and agent $B$ receives the piece between $\kappa(c)$ and $c$. By Remark \ref{inducedivision}, this then induces two cuts on $\Psi$ and a corresponding weighted consensus division of $\Psi$ and therefore of $\mathcal{G}$.

    There are two cases of where these cuts could lie on $\mathcal{G}$ (and respectively on $\Psi$), relative to each other: either (i) the two cuts are on the same connected component, $G_i$ (respectively $\Gamma_i$), or (ii) the two cuts $c$ and $\kappa(c)$ are on different connected components, $G_i$ and $G_j$ respectively (and hence $\Gamma_i$ and $\Gamma_j$). 

    \underline{Case (i):}
    There are $s-1$ connected components $G_k$ with no cut on them; therefore, each of these is assigned to exactly one agent. In this case, the two cuts happened during $\texttt{TT}(\Gamma_i,R_i)$, so we are in the conditions of the proof of Theorem \ref{tree}. Since the tree $\Gamma_i$ has height at most $r+1$, this results in at most $r+2$ connected pieces of $\Gamma_i$ (and hence at most $r+2$ connected pieces of $G_i$) allocated to each agent.
    In this case, there are at most a total of $s+2r+3$ connected pieces to distribute among the two agents.
    
    \underline{Case (ii):}
    There are $s-2$ connected components $G_k$ with no cut on them; therefore, each of these is assigned to exactly one agent.     
    Cut $c$ happens during $\texttt{TT}(\Gamma_i, R_i)$ and cut $\kappa(c)$ happens during $\texttt{TT}(\Gamma_j, R_j)$, so we are in the conditions of Lemma \ref{one_cut}, with each tree having height at most $r+1$. 
    For $\Gamma_i$, this results in an allocation of the at most $r+1$ traced connected pieces of $\Gamma_i$ (and therefore of $G_i$) to agent $A$ and the at most 2 untraced connected pieces of $\Gamma_i$ (and therefore of $G_i$) to agent $B$. 
    For $\Gamma_j$, the at most $r+1$ traced connected pieces of $\Gamma_j$ (and therefore of $G_j$) are allocated to agent $B$ while the at most 2 untraced connected pieces of $\Gamma_j$ (and therefore of $G_j$) are allocated to agent $A$.
    Overall in this case, there are at most a total of $s+2r+4$ connected pieces to distribute among the two agents.
    \end{proof}
\
\section*{Acknowledgements} This work started as part of a final year undergraduate project at the University of Edinburgh in 2022-2023, along with Jiawen Chen, Alyssa Heggison and Aidan Horner; we thank them for their contribution, and especially thank Alyssa for her significant contribution to the work on star cakes. We also thank Nick Sheridan for improving the proof (and statement) of Theorem \ref{thm2},  Warut Suksompong for helpful suggestions, and the anonymous referees for many recommendations that vastly improved an earlier version of this paper.

Ana Rita Pires was partially supported by an Emmy Noether Fellowship from the London Mathematical Society [REF EN-2122-01].


\begin{thebibliography}{ABKR}

\bibitem[A]{Alon}
N.\ Alon, ``Splitting Necklaces'', \textit{Advances in Mathematics} \textbf{63(3)}, (1987): 247–253.

\bibitem[AW]{AlonWest}
N.\ Alon, D.\ West, ``The Borsuk-Ulam theorem and bisection of necklaces'', \textit{Proceedings of the American Mathematical Society} \textbf{98(4)} (1986): 623–628.


\bibitem[Au]{Au:SharingACake}
A.\ Austin, 
``Sharing a Cake.''
\textit{The Mathematical Gazette} \textbf{66(437)} The Mathematical Association (1982): 212-215.

\bibitem[BESS]{BS:divgraphcake2019}
X.\ Bei, E.\ Elkind, E.\ Segal-Halevi, W.\ Suksompong, ``Dividing a Graphical Cake.'' \textit{SIAM Journal on Discrete
Mathematics} \textbf{39(1)} (2025).


\bibitem[BT]{BramsTaylor}
S.\ Brams, A.\ Taylor. ``Fair division: from cake-cutting to dispute resolution'', \textit{Cambridge University Press} (1996).

\bibitem[BTZ]{BTZ}
S.\ Brams, A.\ Taylor, W.\ Zwicker. ``A Moving-Knife Solution to the Four-Person Envy-Free Cake Division'', \textit{Proceedings of the American Mathematical Society} \textbf{125(2)} (1997): 547–554.



\bibitem[F]{stack}
D.\ Fischer, ``Consensus division of a cake.'' \textit{Mathematics Stack Exchange} (2015). \texttt{https://math.stackexchange.com/q/1335267}



\bibitem[GL]{GoLi:Consensus}
P.\ Goldberg, J.\ Li, 
``Consensus Division in an Arbitrary Ratio'',
\textit{14th Innovations in Theoretical Computer Science Conference} , \textbf{251} (2023) 57:1-57:18.



\bibitem[RW]{RoWe:CCA}
J.\ Robertson, W.\ Webb, \emph{Cake-Cutting Algorithms: Be Fair if You Can}, A K Peters Series/ CRC Press (1998).


\bibitem[SZ]{SZ:MarkChooseCut}
H.\ Shishido, DZ.\ Zeng, ``Mark-Choose-Cut Algorithms for Fair and Strongly Fair Division'', \textit{Group Decision and Negotiation} \textbf{8(2)} (1999): 125--137.

\bibitem[SS]{SiSu}
F.\ Simmons, F.\ Su, ``Consensus-halving via theorems of Borsuk-Ulam and Tucker'',
\textit{Mathematical Social Sciences},
\textbf{45(1)} (2003): 15--25.


\bibitem[St]{St:problemofd}
H.\ Steinhaus, ``The problem of fair division'', \textit{Econometrica} \textbf{16(1)} (1948): 101--104.



\bibitem[SW]{SW:StormquistWoodall}
W.\ Stromquist, D.\ Woodall, ``Sets on which several measures agree'', \textit{Journal of Mathematical Analysis} \textbf{108} (1985): 241--248.

\bibitem[YS]{approx}
S.\ Yuen, W.\ Suksompong, ``Approximate Envy-Freeness in Graphical Cake Cutting'', \textit{Proceedings of the Thirty-Second International Joint Conference on Artificial Intelligence} (2023): 2923--2930. 

\end{thebibliography}
\end{document}